\documentclass[a4paper,12pt]{amsart}

\usepackage{amssymb}
\usepackage{latexsym} 
\usepackage{amsfonts} 
\usepackage{amsmath}
\usepackage{eucal} 
\usepackage{bm} 
\usepackage{bbm} 
\usepackage{graphicx} 
\usepackage[english]{varioref} 
\usepackage[nice]{nicefrac} 
\usepackage[all]{xy}
\usepackage{amsthm}

\newcommand{\supp}{\text {\rm supp}}

\newcommand{\Ad}{\text {\rm Ad}}

\def\i{^{-1}}
\def\ge{\geqslant}
\def\le{\leqslant}
\def\<{\langle} 
\def\>{\rangle}

\def\G{\Gamma}
\def\d{\delta}

\def\L{\Lambda}

\def\l{\lambda}

\def\ZZ{\mathbb Z}

\def\CC{\mathbb C}

\def\qq{\mathbf q}

\def\co{\mathcal O}

\def\bH{\mathbb H}

\def\subset{\subseteq}

\theoremstyle{plain}
\newtheorem{thm}{Theorem}[section] 
\newtheorem*{thm*}{Theorem} 
 \newtheorem{prop}[thm]{Proposition}
 \newtheorem{lem}[thm]{Lemma}
 \newtheorem{cor}[thm]{Corollary}

\theoremstyle{definition}

\newtheorem{example}[thm]{Example}

\theoremstyle{remark}

\newtheorem*{rmk}{Remark}
\newtheorem*{claim*}{Claim}

\begin{document}

\author{Xuhua He}
\address{Department of Mathematics, University of Maryland, College Park, MD 20742, USA and department of Mathematics, HKUST, Hong Kong}
\email{xuhuahe@math.umd.edu}
\title{Centers and Cocenters of $0$-Hecke algebras}

\dedicatory{Dedicated to David Vogan on his 60th birthday}

\begin{abstract}
In this paper, we give explicit descriptions of the centers and cocenters of $0$-Hecke algebras associated to finite Coxeter groups. 
\end{abstract}

\maketitle

\section*{Introduction}

\subsection{} Iwahori-Hecke algebras $H_q$ are deformations of the group algebras of finite Coxeter groups $W$ (with nonzero parameters $q$). They play an important role in the study of representations of finite groups of Lie type. 

In 1993, Geck and Pfeiffer \cite{GP93} discovered some remarkable properties of the minimal length elements in their conjugacy classes in $W$ (see Theorem \ref{min-elt}). Based on these properties, they defined the ``character table'' for Iwahori-Hecke algebras. They also gave a basis of the cocenter of Iwahori-Hecke algebras, using minimal length elements. Later, Geck and Rouquier \cite{GR} gave a basis of the center of Iwahori-Hecke algebras. It is interesting that both centers and cocenters of Iwahori-Hecke algebras are closely related to minimal length elements in the finite Coxeter groups and their dimensions both equal the number of conjugacy classes of the finite Coxeter groups. 

\subsection{} The $0$-Hecke algebra $H_0$ was used by Carter and Lusztig in \cite{CL} in the study of $p$-modular representations of finite groups of Lie type. It is a deformation of the group algebras of finite Coxeter groups (with zero parameter). In this paper, we study the center and cocenter of $0$-Hecke algebras $H_0$. We give a basis of the center of $H_0$ in Theorem \ref{0-center} and a basis of the cocenter of $H_0$ in Theorem \ref{main}. 

\subsection{} It is interesting to compare the (co)centers of $H_q$ and $H_0$. Let $W_{\min}$ be the set of minimal length elements in their conjugacy classes in $W$. There are two equivalence relations $\sim$ and $\approx$, on $W_{\min}$ (see \S \ref{setting} for the precise definition). Hence we have the partition of $W_{\min}$ into $\sim$-equivalence classes and $\approx$-equivalence classes. The second partition is finer than the first one. 

The center and cocenter of $H_q$ have basis sets indexed by the set of conjugacy classes of $W$, which are in natural bijection with $W_{\min}/\sim$. The cocenter of $H_0$ has a basis set indexed by $W_{\min}/\approx$ and the center of $H_0$ has a basis set indexed by $W_{\max}/\approx$. Here $W_{\max}/\approx$ is defined using maximal length elements instead and there is a natural bijection between $W_{\max}/\approx$ with the set of $\approx$-equivalence classes of minimal length elements in their ``twisted'' conjugacy classes in $W$. In general, the number of elements in $W_{\max}/\approx$ is different from the number of elements in $W_{\min}/\approx$. 

\subsection{} The paper is organized as follows. In section 1, we recall some properties of the minimal length and maximal length elements. In section 2, we recall the results on the center and cocenter of $H_q$. We give parameterizations of $W_{\min}/\approx$ and $W_{\max}/\approx$ in section 3. In section 4, we give a basis of the center of $H_0$ and in section 5, we give a basis of the cocenter of $H_0$. In section 6, we describe the image of a standard element $t_w$ in the cocenter of $H_0$ and discuss some applications to the class polynomials of $H_q$. 



\section{Finite Coxeter groups}

\subsection{} Let $S$ be a finite set. A Coxeter matrix $(m_{s, s'})_{s, s' \in S}$ is a matrix with entries in $\mathbb N \cup \{\infty\}$ such that $m_{s s}=1$ and $m_{s, s'}=m_{s', s} \ge 2$ for all $s \neq s'$ in $S$. The Coxeter group $W$ associated to the Coxeter matrix $(m_{s, s'})$ is the group generated by $S$ with relations $(s s')^{m_{s, s'}}=1$ for $s, s' \in S$ with $m_{s, s'}< \infty$. The Coxeter group $W$ is equipped with the length function $\ell: W \to \mathbb N$ and the Bruhat order $\le$. 

For any $J \subset S$, let $W_J$ be the subgroup of $W$ generated by elements in $J$. Then $W_J$ is also a Coxeter group. 

\subsection{}\label{setting} Let $\d$ be an automorphism of $W$ with $\d(S)=S$. We say that the elements $w, w' \in W$ are $\d$-conjugate if there exists $x \in W$ such that $w'=x w \d(x) \i$. Let $cl(W)_\d$ be the set of $\d$-conjugacy classes of $W$. We say that a $\d$-conjugacy class $\co$ is {\it elliptic} if $\co \cap W_J=\emptyset$ for any $J=\d(J) \subsetneqq S$. 

For any $w \in W$, let $\supp(w)$ be the set of simple reflections that appear in some (or equivalently, any) reduced expression of $w$. Set $\supp_\d(w)=\cup_{i \ge 0} \d^i(\supp(w))$. Then $\co \in cl(W)_\d$ is elliptic if and only if $\supp_\d(w)=S$ for any $w \in \co$.

For $w, w' \in W$ and $s \in S$, we write $w \xrightarrow{s}_\d w'$ if $w'=s w \d(s)$ and $\ell(w') \le \ell(w)$.  We write $w \to_\d w'$ if there exists a sequence $w=w_0, w_1, \cdots, w_n=w'$ of elements in $W$ such that for any $k$, $w_{k-1} \xrightarrow{s}_\d w_k$ for some $s \in S$. We write $w \approx_\d w'$ if $w \to_\d w'$ and $w' \to_\d w$. 

We say that the two elements $w, w' \in W$ are {\it elementarily strongly $\d$-conjugate} if $\ell(w)=\ell(w')$ and there exists $x \in W$ such that $w'=x w \d(x) \i$, and $\ell(x w)=\ell(x)+\ell(w)$ or $\ell(w \d(x) \i)=\ell(x)+\ell(w)$. We say that $w, w'$ are {\it strongly $\d$-conjugate} if there exists a sequence $w=w_0, w_1, \cdots, w_n=w'$ such that for each $i$, $w_{i-1}$ is elementarily strongly $\d$-conjugate to $w_i$. We write $w \sim_\d w'$ if $w$ and $w'$ are strongly $\d$-conjugate. It is easy to see that %

\begin{lem}
If $w, w' \in W$ with $w \approx_\d w'$, then $w \sim_\d w'$. 
\end{lem}

Note that $\sim_\d$ and $\approx_\d$ are both equivalence relations. For any $\co \in cl(W)$, let $\co_{\min}$ be the set of minimal length elements in $\co$ and $\co_{\max}$ be the set of maximal length elements in $\co$. Since $\sim_\d$ and $\approx_\d$ are compatible with the length function, both $\co_{\min}$ and $\co_{\max}$ are unions of $\sim_\d$-equivalence classes and unions of $\approx_\d$-equivalence classes. 

Let $W_{\d, \min}=\sqcup_{\co \in cl(W)_\d} \co_{\min}$ and $W_{\d, \max}=\sqcup_{\co \in cl(W)_\d} \co_{\max}$. Let $W_{\d, \min}/\sim_\d$ be the set of $\sim_\d$-equivalence classes in $W_{\min}$. We define $W_{\d, \min}/\approx_\d$, $W_{\d, \max}/\sim_\d$ and $W_{\d, \max}/\approx_\d$ in a similar way. 

If $\d$ is the identity map, then we may omit $\d$ in the subscript. 

\

The following result is proved in \cite[Theorem 1.1]{GP93}, \cite[Theorem 2.6]{GKP} and \cite[Theorem 7.5]{He05} (see also \cite{HN12} for a case-free proof). 

\begin{thm}\label{min-elt}
Let $W$ be a finite Coxeter group and $\co$ be a $\d$-conjugacy class of $W$. Then

(1) For any $w \in \co$, there exists $w' \in \co_{\min}$ such that $w \to_\d w'$. 

(2) $\co_{\min}$ is a single strongly $\d$-conjugate class. 

(3) If $\co$ is elliptic, then $\co_{\min}$ is a single $\approx_\d$-equivalence class. 
\end{thm}

As a consequence of Theorem \ref{min-elt}, it is proved in \cite[Corollary 4.5]{He05} that the set of minimal length elements in $\co$ coincides with the set of minimal elements in $\co$ with respect to the Bruhat order $\le$. 

\begin{cor}\label{min}
Let $W$ be a finite Coxeter group and $\co$ be a $\d$-conjugacy class of $W$. Then $\co_{\min}=\{w \in \co; w' \nless w \text{ for any } w' \in \co\}$. 
\end{cor}

\subsection{}\label{min-max} One may transfer the results on minimal length elements to results on maximal length elements via the trick in \cite[\S 2.9]{GKP}. Let $w_0$ be the longest element in $W$ and $\d'=\Ad(w_0) \circ \d$ be the automorphism on $W$. Then the map $$W \to W, w \mapsto w w_0$$ reverses the Bruhat order and sends a $\d$-conjugacy class $\co$ to a $\d'$-conjugacy class $\co'$. Moreover, $w_1 \to_\d w_2$ if and only if $w_2 w_0 \to_{\d'} w_1 w_0$. Thus

\begin{thm}\label{max-elt}
Let $W$ be a finite Coxeter group and $\co$ be a $\d$-conjugacy class of $W$. Then

(1) For any $w \in \co$, there exists $w' \in \co_{\max}$ such that $w' \to_\d w$. 

(2) $\co_{\max}=\{w \in \co; w \nless w' \text{ for any } w' \in \co\}$. 
\end{thm}

\section{Finite Hecke algebras}

In the rest of this paper, we assume that $W$ is a finite Coxeter group. 

\subsection{} Let $\qq$ be an indeterminate and $\L=\CC[\qq]$. The generic Hecke algebra (with equal parameters) $\bH$ of $W$ is the $\L$-algebra generated by $\{T_w; w \in W\}$ subject to the relations:

\begin{enumerate}
\item $T_w \cdot T_{w'}=T_{w w'}$, if $\ell(w w')=\ell(w)+\ell(w')$. 

\item $(T_s+1)(T_s-\qq)=0$ for $s \in S$. 
\end{enumerate}

Given $q \in \CC$, let $\CC_q$ be the $\L$-module where $\qq$ acts by $q$. Let $H_q=\bH\otimes_{\L} \CC_q$ be a specialization of $\bH$.  

In particular, $H_1=\CC[W]$ is the group algebra. The algebra $H_0$ is called the {\it $0$-Hecke algebra}. We will discuss it in details in the next section. 

For any $w \in W$, we denote by $T_{w, q}=T_w \otimes 1 \in H_q$. We simply write $t_w$ for $T_{w, 0}$. 

\subsection{} Let $[\bH, \bH]_\d$ be the $\d$-commutator of $\bH$, that is, the $\L$-submodule of $\bH$ spanned by $[h, h']=h h'-h' \d(h)$ for $h, h' \in \bH$. Let $\overline{\bH}_\d=\bH/[\bH, \bH]_\d$ be the $\d$-cocenter of $\bH$. 

For any $q \in \CC$, we define the $\d$-cocenter $\overline{H_{q, \d}}$ in the same way. Notice that if $q \neq 0$, then $T_{w, q}$ is invertible in $H_q$ for any $w \in W$. However, if $q=0$, then $t_w$ is invertible in $H_q$ if and only if $w=1$. This makes a difference in the study of the cocenter of $H_q$ (for $q \neq 0$) and the cocenter of $H_0$. 

\begin{prop}\label{ww'}
Let $w, w' \in W$. If $w \approx_\d w'$, then the image of $T_w$ and $T_{w'}$ in $\overline{\bH}_\d$ are the same. 
\end{prop}

\begin{proof}
It suffices to prove the case where $w \xrightarrow{s}_\d w'$ and $\ell(w)=\ell(w')$. Without loss of generality, we may assume furthermore that $s w<w$. Then $T_w=T_s T_{s w}$ and $T_{w'}=T_{s w} T_{\d(s)}$. Hence the image of $T_w$ and $T_{w'}$ are the same. 
\end{proof}

For $q \neq 0$, a similar argument shows that if $w \sim_\d w'$, then the image of $T_{w, q}$ and $T_{w', q}$ in $\overline{H_{q, \d}}$ are the same. By Theorem \ref{min-elt} (2), for any $\d$-conjugacy class $\co$ of $W$, $\co_{\min}$ is a single strongly $\d$-conjugacy class. Thus 


\begin{prop}{$($\cite[\S 1]{GP93} and \cite[7.2]{GKP}$)$}
\smallskip

If $q \neq 0$, then for any $\co \in cl(W)_\d$ and $w, w' \in \co_{\min}$, the image of $T_{w, q}$ and $T_{w', q}$ in $\overline{H_{q, \d}}$ are the same. 
\end{prop}
\begin{rmk}
We denote this image by $T_{\co, q}$. 
\end{rmk}

\begin{thm}{$($\cite[\S 1]{GP93} and \cite[Theorem 7.4 (1)]{GKP}$)$}\label{q-basis}
\smallskip

If $q \neq 0$, then $\{T_{\co, q}\}_{\co \in cl(W)_\d}$ form a basis of $\overline{H_{q, \d}}$. 
\end{thm}



\begin{prop}{$($\cite[\S 1.2]{GP93} and \cite[Theorem 7.4 (2)]{GKP}$)$}\label{class}
\smallskip

If $q \neq 0$, then there exists a unique polynomial $f_{w, \co} \in \ZZ[q]$ for any $w \in W$ and $\co \in cl(W)_\d$ such that the image of $T_w$ in $\overline{H_{q, \d}}$ equals $\sum_{\co \in cl(W)_\d} f_{w, \co} T_{\co, q}$. 
\end{prop}
\begin{rmk}
The polynomials $f_{w, \co}$ are called the {\it class polynomials}. They play an important role in the study of characters of Hecke algebras. 
\end{rmk}

\begin{thm}{$($\cite[Theorem 5.2]{GR}$)$}
\smallskip

Let $q \neq 0$. Let $$Z(H_q)_\d=\{h \in H_q; h' h=h \d(h') \text{ for any } h' \in H_q\}$$ be the $\d$-center of $H_q$. For any $\co \in cl(W)_\d$, set $$z_\co=\sum_{w \in W} q^{-\ell(w)} f_{w, \co} T_{w \i}.$$ Then $\{z_\co\}_{\co \in cl(W)_\d}$ form a basis of $Z(H_q)_\d$. 
\end{thm}

\

As a consequence of the results above, we have

\begin{cor}
If $q \neq 0$, then $$\dim Z(H_q)_\d=\dim \overline{H_{q, \d}}=\sharp cl(W)_\d.$$
\end{cor}

\section{Parameterizations of $W_{\d, \min}/\approx_\d$ and $W_{\d, \max}/\approx_\d$}

\subsection{} Notice that for $q \neq 0$, both $\overline{H_{q, \d}}$ and $Z(H_q)_\d$ have basis sets indexed by $cl(W)_\d$, which is in natural bijection with $W_{\d, \min}/\sim_\d$. As we will see later in this paper, for $\overline{H_{0, \d}}$ and $Z(H_0)_\d$, we need to use $W_{\d, \min}/\approx_\d$ and $W_{\d, \max}/\approx_\d$ instead. We give parameterizations of these sets here. 

\subsection{}\label{g-conj} Let $\G_\d=\{(J, C); J=\d(J) \subset S, C \in cl(W_J)_\d \text{ is elliptic}\}$. There is a natural map $$f: \G_\d \to cl(W)_\d, \qquad (J, C) \mapsto \co,$$ where $\co$ is the unique $\d$-conjugacy class of $W$ that contains $C$. 

We say that $(J, C)$ is equivalent to $(J', C')$ if there exists $x \in W^\d$ and the conjugation by $x$ sends $J$ to $J'$ and sends $C$ to $C'$. By \cite[Proposition 5.2.1]{CH}, $f$ induces a bijection from the equivalence classes of $\G_\d$ to $cl(W)_\d$. 

\begin{prop}\label{g-j}
Let $\co \in cl(W)_\d$. Then $$\co_{\min}=\sqcup_{(J, C) \in \G_\d \text{ with } f(J, C)=\co} C_{\min}.$$ 
\end{prop}

\begin{proof}
If $(J, C) \in \G_\d$ with $f(J, C)=\co$, we have $C_{\min} \subset \co_{\min}$ by \cite[Lemma 7.3]{He05}. 

Let $w \in \co_{\min}$. Let $J=\supp_\d(w)$ and $C \in cl(W_J)_\d$  with $w \in C$. By \cite[Theorem 7.5 (P1)]{He05}, $C$ is an elliptic $\d$-conjugacy class of $W_J$. Since $w \in \co_{\min}$ and $w \in C$, $w \in C_{\min}$. 
\end{proof}

\begin{cor}\label{fff}
The map $$f: \G_\d \to W_{\d, \min}/\approx_\d, \qquad (J, C) \mapsto C_{\min}$$ is a bijection. 
\end{cor}

\begin{proof}
Let $(J, C) \in \G_\d$ and $w \in C_{\min}$. If $w \xrightarrow{s}_\d w'$, then $w'=w$ or $s w<w$ or $w \d(s)<w$. In the latter two cases, $s \in J$. Therefore $w' \in C$. Since $w \in C_{\min}$ and $\ell(w') \le \ell(w)$, $w' \in C_{\min}$. 

By definition of $\approx_\d$, $v \in C_{\min}$ for any $v \in W$ with $w \approx_\d v$. On the other hand, by Theorem \ref{min-elt}, $C_{\min}$ is a single $\approx_\d$-equivalence class. Hence the map $(J, C) \mapsto C_{\min} \in W_{\d, \min}/\approx_\d$ is well-defined. 

It is obvious that this map is injective. The surjectivity follows from Proposition \ref{g-j}. 
\end{proof}

\

Using the argument in \S \ref{min-max}, we also obtain

\begin{cor}\label{ff}
Set $\d'=\Ad(w_0) \circ \d$. The map $$\G_{\d'} \to W_{\d, \max}/\approx_\d, \qquad (J, C) \mapsto C_{\min} w_0$$ is a bijection. 
\end{cor}

\begin{example}\label{eg}
Let $W=S_3$. Then $\sharp cl(W)=3$, $\sharp \G=4$ and $\sharp \G_{\Ad(w_0)}=3$. Therefore $\sharp(W_{\min}/\approx) \neq \sharp cl(W)$ and $\sharp(W_{\min}/\approx) \neq \sharp(W_{\max}/\approx)$ for $W=S_3$. 
\end{example}

\section{Centers of $0$-Hecke algebras}

\subsection{} Let $\Sigma \in W_{\d, \max}/\approx_\d$. Set 
\begin{gather*} 
W_{\le \Sigma}=\{x \in W; x \le w \text{ for some } w\in \Sigma\}, \\
t_{\le \Sigma}=\sum_{x \in W_{\le \Sigma}} t_x.
\end{gather*}

Now we recall the following known result on the Bruhat order (see, for example, \cite[Lemma 2.3]{Lu-Hecke}). 

\begin{lem}\label{sxy}
Let $x, y \in W$ with $x \le y$. Let $s \in S$. Then 

(1) $\min\{x, s x\} \le \min\{y, sy\}$ and $\max\{x, s x\} \le \max\{y, s y\}$. 

(2) $\min\{x, x s\} \le \min\{y, y s\}$ and $\max\{x, x s\} \le \max\{y, y s\}$. 
\end{lem}

\begin{lem}\label{sxs}
Let $\Sigma \in W_{\d, \max}/\approx_\d$ and $s \in S$. Then $\{x \in W; x \notin W_{\le \Sigma}, s x \in W_{\le \Sigma}\}=\{x \in W; x \notin W_{\le \Sigma}, x \d(s) \in W_{\le \Sigma}\}$. 
\end{lem}

\begin{proof}
Let $x \in W$ with $x \notin W_{\le \Sigma}, s x \in W_{\le \Sigma}$. By definition, $s x \le w$ for some $w \in \Sigma$. Since $x \nleqslant w$, we have $s x <x$ and $s w>w$ by Lemma \ref{sxy}. Thus $\ell(s w \d(s)) \ge \ell(s w)-1=\ell(w)$. Since $w \in W_{\d, \max}$, $\ell(s w \d(s))=\ell(w)$ and $s w s \in \Sigma$. Moreover, $s w \d(s)<s w$. 

Since $s x \le w$ and $w<s w$, $x \le s w$. By Lemma \ref{sxy}, $\min\{x, x \d(s)\} \le s w \d(s)$. Since $x \notin W_{\le \Sigma}$, $x \d(s) \in W_{\le \Sigma}$. 
\end{proof}

\begin{lem}\label{t-center}
Let $\Sigma \in W_{\d, \max}/\approx_\d$. Then $t_{\le \Sigma} \in Z(H_0)_\d$. 
\end{lem}

\begin{proof}
Let $s \in S$. Then $$t_s t_{\le \Sigma}=\sum_{x \in W_{\le \Sigma}} t_s t_x=\sum_{x, s x \in W_{\le \Sigma}} t_s t_x+\sum_{y \in W_{\le \Sigma}, s y \notin W_{\le \Sigma}} t_s t_x.$$ 

If $x, s x \in W_{\le \Sigma}$, then $t_s t_x+t_s t_{s x}=0$. If $y \in W_{\le \Sigma}, s y \notin W_{\le \Sigma}$, then $y<s y$ and $t_s t_y=t_{s y}$. Therefore $$t_s t_{\le \Sigma}=\sum_{x \in W; x \notin W_{\le \Sigma}, s x \in W_{\le \Sigma}} t_x.$$

Similarly, $$t_{\le \Sigma} t_{\d(s)}=\sum_{x \in W; x \notin W_{\le \Sigma}, x \d(s) \in W_{\le \Sigma}} t_x.$$

By Lemma \ref{sxs}, $t_s t_{\le \Sigma}=t_{\le \Sigma} t_{\d(s)}$ for any $s \in S$. Thus $t_{\le \Sigma} \in Z(H_0)_\d$. 
\end{proof}

\begin{thm}\label{0-center}
The elements $\{t_{\le \Sigma}\}_{\Sigma \in W_{\d, \max}/\approx_\d}$ form a basis of $Z(H_0)_\d$. 
\end{thm}

\begin{proof} For any $h=\sum_{w \in W} a_w t_w \in H_0$, we write $\supp(h)=\{w \in W; a_w \neq 0\}$. Let $\supp(h)_{\max}$ be the set of maximal length elements in $\supp(h)$. We show that 

(a) If $h \in Z(H_0)_\d$ and $w \in \supp(h)_{\max}$, then $s w \d(s) \in \supp(h)_{\max}$ and $a_{s w \d(s)}=a_w$ for any $s \in S$ with $s w>w$ or $w s>w$.  

Without loss of generality, we assume that $s w>w$. Then $s w \in \supp(t_s h)=\supp(h t_{\d(s)})$ and 
\begin{gather*} \supp(t_s h)_{\max}=\{s x; x \in \supp(h)_{\max}, s x>x\}, \\
\supp(h t_{\d(s)})_{\max}=\{y \d(s); y \in \supp(h)_{\max}, y \d(s)>y\}.
\end{gather*}

Therefore $s w \d(s) \in \supp(h)_{\max}$ and $\ell(s w \d(s))=\ell(w)$. The coefficient of $t_{s w}$ in $t_s h$ is $a_w$ and the coefficient of $t_{s w}=t_{(s w \d(s)) \d(s)}$ in $h t_{\d(s)}$ is $a_{s w \d(s)}$. Thus $a_w=a_{s w \d(s)}$. 

(a) is proved. 

Now we show that 

(b) If $h \in Z(H_0)_\d$, then $\supp(h)_{\max} \subset W_{\d, \max}$. 

If $w \notin W_{\d, \max}$, then by Theorem \ref{max-elt}, there exists $w'$ with $\ell(w')=\ell(w)+2$ and $s \in S$ with $w' \to_\d s w' \d(s) \approx_\d w$. By (a), $s w' \d(s) \in \supp(h)_{\max}$ since $s w' \d(s) \approx_\d w$. Since $s w'<w'$, by (a) again, $w' \in \supp(h)_{\max}$. That is a contradiction. 

(b) is proved. 

Now suppose that $\oplus_{\Sigma \in W_{\d, \max}/\approx_\d} \CC t_{\le \Sigma} \subsetneqq Z(H_0)_\d$. Let $h$ be an element in $Z(H_0)_\d - \oplus_{\Sigma \in W_{\d, \max}/\approx_\d} \CC t_{\le \Sigma}$ and $\max_{w \in \supp(h)} \ell(w) \le \max_{w \in \supp(h')} \ell(w)$ for any $h' \in Z(H_0)_\d - \oplus_{\Sigma \in W_{\d, \max}/\approx_\d} \CC t_{\le \Sigma}$. 

By (a) and (b), $\supp(h)_{\max}$ is a union of $\Sigma$ with $\Sigma \in W_{\d, \max}/\approx_\d$. By (a), if $\Sigma \subset \supp(h)_{\max}$, then $a_w=a_{w'}$ for any $w, w' \in \Sigma$. We set $a_\Sigma=a_w$ for any $w \in \Sigma$. Set $h'=h-\sum_{\Sigma \subset \supp(h)_{\max}} a_\Sigma t_{\le \Sigma}$. Then $h' \in Z(H_0)_\d - \oplus_{\Sigma \in W_{\d, \max}/\approx_\d} \CC t_{\le \Sigma}$. But $\max_{w \in \supp(h')} \ell(w)<\max_{w \in \supp(h)} \ell(w)$. That is a contradiction. 
\end{proof}

\subsection{} In fact, Theorem \ref{0-center} also holds for the $0$-Hecke algebras associated to any affine Weyl group and the proof is similar (the only difference is that one use \cite[Main Theorem 1.1]{Ro} instead of Theorem \ref{max-elt}). 

On the other hand, there are other explicit descriptions of the centers of finite and affine Hecke algebras. 
\begin{itemize}
\item Geck and Rouquier \cite[Theorem 5.2]{GR} gave a basis of the centers of finite Hecke algebras with parameter $q \neq 0$. 

\item Bernstein, and Lusztig \cite[Proposition 3.11]{Lu89} gave a basis of the centers of affine Hecke algebras with parameter $q \neq 0$. 

\item Vign\'eras \cite[Theorem 1.2]{Vi} gave a basis of the centers of affine $0$-Hecke algebras and pro-$p$ Hecke algebras. 
\end{itemize}

It is interesting to compare Theorem \ref{0-center} (for finite and affine $0$-Hecke algebras) with the above results. 

\section{Cocenters of $0$-Hecke algebras}

\subsection{} For any $\Sigma \in W_{\d, \min}/\approx_\d$, we denote by $T_{\Sigma}$ the image of $T_w$ in $\overline{\bH}_\d$ for any $w \in \Sigma$. By Proposition \ref{ww'}, the element $T_{\Sigma}$ is well-defined. Similar to the proof of Theorem \ref{q-basis}, we have 


\begin{prop}\label{q-0-span}
The set $\{T_{\Sigma}\}_{\Sigma \in W_{\d, \min}/\approx_\d}$ spans $\overline{\bH}_\d$. 
\end{prop}

\subsection{} 
Via the natural bijection $f: \G_\d \to W_{\d, \min}/\approx_\d$ in Corollary \ref{fff}, we may write $T_{(J, C)}$ for $T_{f(J, C)}$. We also write $t_{(J, C)}=t_{f(J, C)}$ for $T_{(J, C)} \otimes 1 \in \overline{H_{0, \d}}=\overline{\bH}_\d \otimes_\L \CC_0$. 

It is worth mentioning that $\overline{\bH}_\d$ is not a free module over $\L$ by Theorem \ref{q-basis} and Theorem \ref{main} we will prove later. This is because $\dim \overline{H_{q, \d}}=\sharp cl(W)_\d$ for any $q \neq 0$ and $\dim \overline{H_{0, \d}}=\sharp W_{\d, \min}/\approx_\d$. These numbers do not match in general (see Example \ref{eg}). 

\subsection{} Now we come to the cocenter of $0$-Hecke algebras. 

We first recall the Demazure product. 

By \cite{He08}, for any $x, y \in W$, the set $\{u v; u \le x, v \le y\}$ contains a unique maximal element. We denote this element by $x*y$ and call it the {\it Demazure product} of $x$ and $y$. It is easy to see that $\supp(x*y)=\supp(x) \cup \supp(y)$. The following result is proved in \cite[Lemma 1]{He08}.

\begin{lem}\label{5-2}
Let $x, y \in W$. Then $$t_x t_y=(-1)^{\ell(x)+\ell(y)-\ell(x*y)} t_{x*y}.$$
\end{lem}




\begin{lem}\label{grading}
For any $J=\d(J) \subset S$, set $H^{\supp_\d=J}_0 =\oplus_{\supp_\d(w)=J} \CC t_w$. Then $$[H_0, H_0]_\d=\oplus_{J=\d(J) \subset S} \bigl(H^{\supp_\d=J}_0 \cap [H_0, H_0]_\d \bigr).$$
\end{lem}

\begin{proof}
By Lemma \ref{5-2}, for any $x, y \in W$, 
\begin{gather*}
t_x t_y=(-1)^{\ell(x)+\ell(y)-\ell(x*y)} t_{x*y}, \\
t_y t_{\d(x)}=(-1)^{\ell(x)+\ell(y)-\ell(y*(\d(x))} t_{y*\d(x)}.
\end{gather*}

Moreover, $\supp_\d(x*y)=\supp_\d(x) \cup \supp_\d(y)=\supp_\d(y*(\d(x))$. Thus $t_x t_y, t_y t_{\d(x)} \in H^{\supp_\d=\supp_\d(x*y)}_{0}$ and $t_x t_y-t_y t_{\d(x)} \in H^{\supp_\d=\supp_\d(x*y)}_{0}$. 
\end{proof}

\

Another result we need here is that the elliptic conjugacy classes never ``fuse''. 

\begin{thm}{$($\cite[Theorem 3.2.11]{GP00} and \cite[Theorem 5.2.2]{CH}$)$}\label{fuse}\footnote{The proof in \cite{GP00} and \cite{CH} are based on a characterization of elliptic conjugacy classes using characteristic polynomials \cite[Theorem 3.2.7 (P3)]{GP00} and \cite[Theorem 7.5 (P3)]{He05}, which is proved via a case-by-case analysis. It is interesting to find a case-free proof of these results. }
\smallskip

Let $J=\d(J) \subset S$. Let $C, C'$ be two distinct elliptic $\d$-conjugacy classes of $W_J$. Then $C$ and $C'$ are not $\d$-conjugate in $W$. 
\end{thm}

\

Now we come to the main theorem of this section. 

\begin{thm}\label{main}
The elements $\{t_{(J, C)}\}_{(J, C) \in \G_\d}$ form a basis of $\overline{H_{0, \d}}$. 
\end{thm}






\begin{proof}
Suppose that $\sum_{(J, C) \in \G_\d} a_{(J, C)} t_{(J, C)}=0$ in $\overline{H_{0, \d}}$ for some $a_{(J, C)} \in \CC$. 
Then by Lemma \ref{grading}, for any $J=\d(J) \subset S$, $$\sum_{C \in cl(W_J)_\d \text{ is elliptic}} a_{(J, C)} t_{(J, C)}=0.$$ 

Fix $J=\d(J) \subset S$. We show that 

(a) The set $\{T_{(J, C)}\}_{C \in cl(W_J)_\d \text{ is } elliptic}$ is a linearly independent set in $\overline{\bH}_\d$. 

Suppose that \[\sum_{C \in cl(W_J)_\d \text{ is } elliptic} b_C T_{(J, C)}=0 \in \overline{\bH}_\d\] for some $b_C \in \Lambda$. Then \[\sum_{C \in cl(W_J)_\d \text{ is } elliptic} b_C \mid_{\qq=q}  T_{(J, C)}=0 \in \overline{H_{q, \d}}\] for any $q \neq 0$. By Theorem \ref{q-basis}, the set $\{T_{(J, C), q}\}_{C \in cl(W_J)_\d \text{ is } elliptic}$ is a linearly independent set in $\overline{H_{q, \d}}$ for any $q \neq 0$. Hence $b_C \mid_{\qq=q}=0$ for any $q \neq 0$. Thus $b_C=0$. 

(a) is proved. 

In other words, $\sum_{C \in cl(W_J)_\d \text{ is } elliptic} \Lambda T_{(J, C)}$ is a free submodule of $\overline{\bH}$ with basis $T_{(J, C)}$. 
Thus $\sum_{C \in cl(W_J)_\d \text{ is } elliptic} \CC t_{(J, C)}$ is a free submodule of $\overline{H_{0, \d}}$ with basis $t_{(J, C)}$.  Therefore $a_{J, C}=0$.
\end{proof}

\subsection{} Now we relate the cocenter of $H_0$ to the representations of $H_0$. 

For any $J \subset S$, let $\l_J$ be the one-dimensional representation of $H_0$ defined by 
\[
\l_J(t_s)=\begin{cases} -1, & \text{ if } s \in J; \\ 0, & \text{ if } s \notin J. \end{cases}
\]

By \cite{No}, the set $\{\l_J\}_{J \subset S}$ is the set of all the irreducible representations of $H_0$. 

Let $R(H_0)$ be the Grothendieck group of finite dimensional representations of $H_0$. Then $R(H_0)$ is a free group with basis $\{\l_J\}_{J \subset S}$. Consider the trace map $$Tr: \overline{H_0} \to R(H_0)^*, \qquad h \mapsto (V \mapsto tr(h, V)).$$

It is easy to see that for any $(J, C) \in \G$ and $K \subset S$, 
\[
tr(t_{J, C}, \l_K)=\begin{cases} (-1)^{\ell(C)}, & \text{ if } J \subset K; \\ 0, & \text{ otherwise}. \end{cases}
\]

Here $\ell(C)$ is the length of any minimal length element in $C$. 

By \cite[Proposition 6.10]{Lu84}, for any $J \subset S$ and any two elliptic conjugacy classes $C$ and $C'$ of $W_J$, $\ell(C) \equiv \ell(C') \mod 2$. Therefore, 

\begin{prop}
The trace map $Tr: \overline{H_0} \to R(H_0)^*$ is surjective and the kernel equals $\oplus_{J \subset S, C, C' \in cl(W_J) \text{ are elliptic}} \CC \{t_{(J, C)}-t_{(J, C')}\}$.
\end{prop}

\section{A partial order on $W_{\d, \min}/\approx_\d$}

\subsection{} Let $w \in W$ and $\Sigma \in W_{\d, \min}/\approx_\d$, we write $\Sigma \preceq w$ if there exists $w' \in \Sigma$ with $w' \le w$. For $w \in W$ and $\co \in cl(W)_\d$, we define $\co \preceq w$ in the same way. 

We define a partial order on $W_{\d, \min}/\approx_\d$ as follows. 

For $\Sigma, \Sigma' \in W_{\d, \min}/\approx_\d$, we write $\Sigma' \preceq \Sigma$ if $\Sigma' \preceq w$ for some $w \in \Sigma$. By \cite[Corollary 4.6]{He05}, $\Sigma' \preceq \Sigma$ if and only if $\Sigma' \preceq w$ for any $w \in \Sigma$. In particular, $\preceq$ is transitive. This defines a partial order on $W_{\d, \min}/\approx_\d$. 

We define a partial order on $cl(W)_\d$ in a similar way. 

\begin{prop}
Let $\co, \co' \in cl(W)_\d$. The following conditions are equivalent:

(1) For any $w \in \co_{\min}$, there exists $w' \in \co'_{\min}$ such that $w' \le w$. 

(2) There exists $w \in \co_{\min}$ and $w' \in \co'_{\min}$ such that $w' \le w$. 
\end{prop}

\begin{rmk}
We write $\co' \preceq \co$ if the conditions above are satisfied. Then the map $W_{\d, \min}/\approx_\d \to cl(W)_\d$ is compatible with the partial orders $\preceq$. 
\end{rmk}

\begin{proof}
Let $w, w_1 \in \co_{\min}$ and $w' \in \co'_{\min}$ with $w' \le w$. Let $J=\supp_\d(w)$, $J_1=\supp_\d(w_1)$ and $J'=\supp_\d(w')$. Let $C \in cl(W_J)_\d$ and $C_1 \in cl(W_{J_1})_\d$ with $w \in C$ and $w'_1 \in C_1$. By \S \ref{g-conj}, there exists $x \in W^\d$ and the conjugation of $x$ sends $J$ to $J_1$ and sends $C$ to $C_1$. Since $w' \le w$, $J' \subset J$. As the conjugation by $x$ sends simple reflections in $J$ to simple reflections in $J_1$, we have $x w' x \i \le x w x \i$. Moreover, $x w x \i \in C_1$ is a minimal length element. By Theorem \ref{min-elt}, $x w x \i \approx_\d w'$. By \cite[Lemma 4.4]{He05}, there exists $w'_1 \in \co'_{\min}$ with $w'_1 \le w_1$. 
\end{proof}

\begin{prop}\label{sigma-w} Let $w \in W$. Then 

(1) The set $\{\Sigma \in W_{\d, \min}/\approx_\d; \Sigma \preceq w\}$ contains a unique maximal element $\Sigma_w$. 

(2) The image of $t_w$ in $\overline{H_{0, \d}}$ equals $(-1)^{\ell(w)-\ell(\Sigma_w)} t_{\Sigma_w}$. 
\end{prop}

\begin{rmk}
By Theorem \ref{main}, part (2) of the Proposition gives another characterization of $\Sigma_w$.
\end{rmk}

\begin{proof} We argue by induction on $\ell(w)$. 

If $w \in W_{\d, \min}$, we denote by $\Sigma_w$ the $\approx_\d$-equivalence class that contains
$w$. By definition, for any $\Sigma \in W_{\d, \min}/\approx_\d$ with $\Sigma \preceq w$, $\Sigma \preceq \Sigma_w$. Also by definition, the image of $t_w$ in $\overline{H_{0, \d}}$ is $t_{\Sigma_w}$. 

Now suppose that $w \in W_{\d, \min}$. By Theorem \ref{min-elt} (1), there exists $w' \in W$ and $s \in S$ such that $w \approx w'$ and $\ell(s w' \d(s))<\ell(w')$. Let $\Sigma \in W_{\d, \min}/\approx_\d$ with $\Sigma \preceq w$. By [7,Lemma4.4], $\Sigma \preceq w'$. Inother words, there exists $x \in \Sigma$ with $x \le w'$. 

Now we prove that

(a) $\Sigma \preceq \Sigma_{s w'}$. 

If $x<s x$, then by Lemma \ref{sxy}, $x \le s w'$ and $\Sigma \preceq s w'$. 

If $s x<x$, then $\ell(s x \d(s)) \le \ell(s x)+1=\ell(x)$. Hence $s x \d(s) \in \Sigma$. By
Lemma \ref{sxy}, $s x \le s w'$. Since $s w' \d(s)<s w'$, by Lemma \ref{sxy} again, we have $s x \d(s) \le s w'$. Thus $\Sigma \preceq s w'$. 

Since $\ell(s w')<\ell(w)$, by inductive hypothesis, $\Sigma_{s w'}$ is defined and $\Sigma \preceq \Sigma_{s w'}$. 

(a) is proved.

Since $\Sigma_{s w'} \preceq s w'$, $\Sigma_{s w'} \preceq w'$. By [7, Lemma 4.4],  $\Sigma_{s w'} \preceq w$. Thus $\Sigma_{s w'}$ is the unique maximal element in $\{\Sigma \in W_{\d, \min}/\approx_\d; \Sigma \preceq w\}$. 

We also have
$$t_w \equiv t_{w'} \equiv t_s t_{s w'}=t_{s w'} t_{\d(s)}=-t_{s w'} \mod [H_0, H_0]_\d.$$

By inductive hypothesis, the image of $t_{s w'}$ in $\overline{H_{0, \d}}$ is $(-1)^{\ell(s w')-\ell(\Sigma_{s w'})} t_{\Sigma_{s w'}}$. Hence the image of $t_w$ in $\overline{H_{0, \d}}$ is $(-1)^{\ell(w)-\ell(\Sigma_{s w'})} t_{\Sigma_{s w'}}$. 

\subsection{} For any $w \in W$, we denote by $\co_w$ the image of $\Sigma_w$ under the map $W_{\d, \min}/\approx_\d \to cl(W)_\d$. Then $\co_w$ is the maximal element in $\{\co \in cl(W)_\d; \co \preceq w\}$.
 
Now we discuss some application to class polynomials. 
 
Let $w \in W$. By Proposition \ref{class}, for any $q \neq 0$, $$T_{w, q}=\sum_{\co \in cl(W)_\d} f_{w, \co} T_{\co, q} \in \overline{H_{q, \d}}.$$

By the same argument as in Proposition \ref{sigma-w}, $f_{w, \co}=0$ unless $\co \preceq \co_w$. 

Moreover, by Proposition \ref{q-0-span}, there exists $a_{w, \Sigma} \in \L$ such that
$$T_w=\sum_{\Sigma \in W_{\d, \min}/\approx_\d} a_{w, \Sigma} T_{\Sigma} \in \overline{\bH}_\d.$$
 
Let $p: W_{\d, \min}/\approx_\d \to cl(W)_\d$ be the natural map. Then for any $q \neq 0$, $$T_{w, q}=\sum_{\Sigma \in W_{\d, \min}/\approx_\d} a_{w, \Sigma} \mid_{\qq=q} T_{p(\Sigma), q} \in \overline{H_{q, \d}}.$$
 
Therefore for any $\co \in cl(W)_\d$, $\sum_{p(\Sigma)=\co} a_{w, \Sigma}=f_{w, \co}$. 

By Proposition \ref{sigma-w}, $$a_{w, \Sigma} \in \begin{cases} (-1)^{\ell(w)-\ell(\Sigma_w)}+\qq \L, & \text{ if } \Sigma=\Sigma_w; \\ \qq \L, & \text{ otherwise}. \end{cases}$$
 
Therefore $$f_{w, \co} \in \begin{cases} (-1)^{\ell(w)-\ell(\Sigma_w)}+q \ZZ[q], & \text{ if } \Sigma_w \subset \co; \\ q \ZZ[q], & \text{ otherwise}. \end{cases}$$
\end{proof}
 
\section*{Acknowledgement}
We thank D. Ciubotaru, G. Lusztig and S. Nie for helpful discussions. We thank the referee for his/her valuable suggestions.

\end{document}